\documentclass{amsart}

\usepackage{amsmath}
\usepackage[all,cmtip,2cell]{xy}
\usepackage{amssymb}
\usepackage{fullpage}

\newtheorem{global-theorem}{Theorem}
\newtheorem{theorem}{Theorem}[section]
\newtheorem{lemma}[theorem]{Lemma}
\newtheorem{corollary}[theorem]{Corollary}
\newtheorem{proposition}[theorem]{Proposition}
\theoremstyle{definition}
\newtheorem{definition}[theorem]{Definition}
\newtheorem{example}[theorem]{Example}
\newtheorem{remark}[theorem]{Remark}

\usepackage[mathscr]{euscript}

\usepackage{microtype}

\usepackage{url}
\usepackage{tikz}
\usetikzlibrary{decorations.markings}
\usetikzlibrary{arrows}

\newcommand{\dG}{\Delta \mathfrak{G}}

\newcommand{\dC}{\Delta \mathfrak{C}}
\newcommand{\dR}{\Delta \mathfrak{R}}
\newcommand{\dD}{\Delta \mathfrak{D}}
\newcommand{\dQ}{\Delta \mathfrak{Q}}

\begin{document}

\title{Homotopy of Planar Lie Group Equivariant Presheaves}
\author{Scott Balchin}
\address{Department of Mathematics\\
University of Leicester\\
University Road, Leicester LE1 7RH, England, UK}
\email{slb85@le.ac.uk}

\begin{abstract}
We utilise the theory of crossed simplicial groups to introduce a collection of local Quillen model structures on the category of simplicial presheaves with a compact planar Lie group action on a small Grothendieck site.   As an application, we give a characterisation of equivariant cohomology theories on a site as derived mapping spaces in  these model categories.
\end{abstract}

\maketitle

\vspace{-5mm}
	
\section{Introduction}

The cyclic category of Connes, which appears in the theory of non-commutative geometry, has proven to be a very useful tool since its introduction.  Many authors have noted how cyclic objects model topological spaces with a circle action.  This was formally developed by Dwyer, Hopkins and Kan \cite{homotopycc}, Spali\'{n}ski \cite{spalinskistrong} and Blumberg \cite{blumberg} with Quillen equivalences between certain Quillen models on $\textbf{Top}^{SO(2)}$ and cyclic sets (in fact a slight modification of this category in the case of Blumberg).  The cyclic category is a special case of a wider class of categories called crossed simplicial groups \cite{loday}. In particular, it is an example of a planar crossed simplicial group.  Planar crossed simplicial groups are those which can be used to model topological spaces with a planar Lie group action.  For each planar crossed simplicial group, there is a choice of three different canonical models we can equip it with, depending on how much we want to respect the fixed point subgroups of the topological group.

Cyclic sets were used in the theory of cyclic presheaves in the work of Seteklev and {\O}stv{\ae}r \cite{cyclicsheaves} (adapted from the Masters thesis of Seteklev \cite{seteklevmaster}), where a local model structure is developed which mimics the construction of the local model on simplicial presheaves.  This is the same model which is used in the theory of homotopical algebraic geometry \cite{toen2}.  We extend this model structure to all planar crossed simplicial groups.  By construction, the local presheaf model structures that we develop here can be seen as giving the theory of equivariant $(\infty,1)$-stacks.  We use this viewpoint, and present an application for equivariant $(\infty,1)$-stacks as a setting for equivariant cohomology theories.

\subsection*{Outline of the Paper}

We begin by introducing the theory of planar crossed simplicial groups before recalling a class of model structures on their presheaf categories.  We then use these to develop local models on the categories of presheaves valued in these categories using left Bousfield localisation.  Finally, we define equivariant cohomology theories as derived mapping spaces in these model categories.  We will argue that our definition is well founded by exploring the cyclic case.

\section{Crossed Simplicial Groups}\label{sec2}
\subsection{Definition and Classification Theorem}
 Crossed simplicial groups were introduced by Loday and Fiedorowicz    \cite{loday} (and independently by Krasauskas under the name of \textit{skew-simplicial sets} \cite{krasauskas}).  Detailed accounts of the properties of these objects can be found in \cite{surface}.  Here we will recall the necessary notions that we need for our intended application.

\begin{definition}
A \textit{crossed simplicial group} is a category $\Delta \mathfrak{G}$ equipped with an embedding $i \colon \Delta \hookrightarrow \Delta \mathfrak{G}$ such that:
\begin{enumerate}
\item The functor $i$ is bijective on objects.
\item Any morphism $u\colon i[m] \to i[n]$ in $\Delta \mathfrak{G}$ can be uniquely written as $i(\phi) \circ g$ where $\phi \colon [m] \to [n]$ is a morphism in $\Delta$ and $g$ is an automorphism of $i[m]$ in $\Delta \mathfrak{G}$.  We call this decomposition the \textit{canonical decomposition}.
\end{enumerate}
\end{definition}

\begin{example}
The cyclic category of Connes (see    \cite{connes1} or    \cite{connes2}) is an example of a crossed simplicial group where $\mathfrak{G}_n = \mathbb{Z}/(n+1)$ with generator $\tau_{n}$ such that $(\tau_{n})^{n+1} = \text{id}_n$.   We will denote this category $\Delta \mathfrak{C}$.  This category is used, among other places, in the theory of cyclic homology   \cite{loday1998cyclic}.
\end{example}

\begin{example}
From the previous example we can define  new crossed simplicial group.  Let $N > 1$ be a natural number, define $\dC_N$ by setting $\mathfrak{G}_n =  \mathbb{Z}/N(n+1)$.  This is an example of a \textit{planar Lie group} and will be studied in more detail in the following section.
\end{example}

We can generate many more examples of crossed simplicial groups by using the following classification theorem.

\begin{theorem}[{\cite[Theorem 3.6]{loday}}]
Any crossed simplicial group $\Delta \mathfrak{G}$ splits uniquely as a sequence of functors
$$\Delta \mathfrak{G}_n' \to \Delta \mathfrak{G}_n \to \Delta \mathfrak{G}_n''$$
such that $\Delta \mathfrak{G} ' $ is a simplicial group and $\Delta \mathfrak{G}''$ is one of the following seven crossed simplicial groups:

\begin{itemize}
\item $\Delta$  - The \textit{trivial} crossed simplicial group.
\item $\Delta \mathfrak{C}$ -  The \textit{cyclic} crossed simplicial group, $\mathfrak{C}_n =  \mathbb{Z}/(n+1)$.
\item $\Delta \mathfrak{S}$ -  The \textit{symmetric} crossed simplicial group, $\mathfrak{S}_n =  S_{n+1}$.
\item $\Delta \mathfrak{R}$ - The \textit{reflexive} crossed simplicial group, $\mathfrak{R}_n = \mathbb{Z}/2\mathbb{Z}$.
\item $\Delta \mathfrak{D}$  - The \textit{dihedral} crossed simplicial group, $\mathfrak{D}_n = D_{n+1}$.
\item $\Delta \mathfrak{T}$  - The \textit{reflexosymmetric} crossed simplicial group, $\mathfrak{T}_n = T_{n+1} = \mathbb{Z}/2\mathbb{Z} \ltimes S_{n+1}$.
\item $\Delta \mathfrak{W}$ - The \textit{Weyl crossed simplicial group}, $\mathfrak{W}_n = W_{n+1} = \mathbb{Z}/2\mathbb{Z} \wr S_{n+1} $.
\end{itemize}
We will call these crossed simplicial groups the \textit{simple crossed simplicial groups}.  We will say that $\dG$ is of \textit{type} $\dG''$.
\end{theorem}

\begin{example}
An example that arises form the classification theorem uses the quaternion groups $Q_n$.  These groups are dicyclic, therefore we can see that there is a crossed simplicial group $\Delta \mathfrak{Q} = \{Q_{n+1}\}_{n \geq 0}$ due to the extension $\Delta \mathfrak{R} \to \Delta \mathfrak{Q} \to \Delta \mathfrak{D}$. The dihedral (resp., quaternionic) crossed simplicial groups are the object of study in dihedral (resp., quaternionic) homology as discussed in    \cite{dihedral} and    \cite{Spaliński2000557}.
\end{example}

As with the simplex category, our interest with crossed simplicial groups lies in their presheaf categories.  We will now introduce this concept and see how this notion is an extension of simplicial objects.

\begin{definition}
Let $\Delta \mathfrak{G}$ be a crossed simplicial group and $\mathscr{C}$ be a category.  A \textit{$\Delta \mathfrak{G}$-object in $\mathscr{C}$} is defined to be a functor $X \colon (\Delta \mathfrak{G})^{op} \to \mathscr{C}$.  We shall denote such a functor as $X_\ast$ with $X_n$ being the image of $[n]$.  If $\lambda \colon [m] \to [n]$ is a morphism in $\Delta \mathfrak{G}$ we shall denote by $\lambda^\ast \colon X_n \to X_m$ the morphism $X (\lambda)$.  We shall denote the category of all $\Delta \mathfrak{G} \text{-} \mathscr{C}$ objects as $\Delta \mathfrak{G} \text{-} \mathscr{C}$.
\end{definition}

\begin{remark}
The inclusion functor $i \colon \Delta \hookrightarrow \dG$ induces a restriction $i^\ast \colon  \dG \textbf{-Set} \to \textbf{sSet}$ which forgets the $\mathfrak{G}_n$ actions.  This functor is faithful but not full.
\end{remark}

Sometimes it is more convenient to consider a $\Delta \mathfrak{G}$-object as a simplicial object with some extra structure.

\begin{proposition}[{\cite[Lemma 4.2]{loday}}]\label{gobject}
A $\Delta \mathfrak{G}$-object in a category $\mathscr{C}$ is equivalent to a simplicial object $X_\ast$ in $\mathscr{C}$ with the following additional structure:

\begin{itemize}
\item Left group actions $\mathfrak{G}_n \times X_n \to X_n$.
\item Face relations $d_i(gx) = d_i(g)(d_{g^{-1}(i)}x)$.
\item Degeneracy relations $s_i(gx) = s_i(g)(s_{g^{-1}(i)\textbf{•}}x)$.
\end{itemize}

In particular, a $\Delta \mathfrak{G}$-map $f_\ast \colon X_\ast \to Y_\ast$ is the same thing as a simplicial map such that each of the $f_n \colon X_n \to Y_n$ is $\mathfrak{G}_n$-equivariant.
\end{proposition}

\subsection{Planar Crossed Simplicial Groups}

We will now move our attention to a special class of crossed simplicial groups.  Planar crossed simplicial groups were first studied in \cite{surface} and further in \cite{cstft}.  They are important as they correspond to Lie groups that appear as structure groups of surfaces.  First of all we give the notion of the geometric realisation of a crossed simplicial group.

\begin{theorem}[{\cite[Theorem 5.3]{loday}}]\label{real}\leavevmode
\begin{enumerate}
\item Let $\Delta \mathfrak{G}$ be a crossed simplicial group.  Then the realisation $|\mathfrak{G}|$ of the simplicial set $\mathfrak{G} = \{\mathfrak{G}_n\}_{n \in \mathbb{N}}$ is a topological group.
\item The geometric realisation $|N \Delta \mathfrak{G}|$ of the nerve of the category $\Delta \mathfrak{G}$ is homotopy equivalent to the classifying space $B|\mathfrak{G}|$.
\item Let $X \colon \Delta \mathfrak{G}^{op} \to \textbf{Set}$ be a $\Delta \mathfrak{G} \text{-Set}$, then $|i^\ast X|$ has a natural $|\mathfrak{G}|$-action. 
\end{enumerate}
\end{theorem}

\begin{definition}
A \textit{planar crossed simplicial group} is a crossed simplicial group which is of type $\dC$ or $\Delta \mathfrak{D}$ in the classification theorem.
\end{definition}

Planar crossed simplicial groups are an interesting thing to study as their geometric realisations are planar Lie groups, which can be used to add extra structure to surfaces as done in \cite{surface}.

\begin{definition}
A morphism $\rho \colon \widetilde{G} \to G$ is a \textit{connective covering} if $\rho$ is a covering of its image and $\rho^{-1}(G_e)$ coincides with $\widetilde{G}_e$.  A \textit{planar Lie group} is a connective covering of $O(2)$.
\end{definition}

Planar Lie groups have a full classification, and this classification matches up exactly with the planar crossed simplicial groups.

\begin{definition}
Table \ref{cstable} gives all of the planar crossed simplicial groups along with their geometric realisations, where we consider all $M,N \in \mathbb{N}$.  The groups $\widetilde{SO(2)}$ and $\widetilde{O(2)}$ are the universal covers of $SO(2)$ and $O(2)$ respectively.

\begin{table}[ht!]
\centering
\begin{tabular}{|c|c|c|c|}
\hline
$\mathbf{\Delta \mathfrak{G}}$ & $\mathfrak{G}_n$ &$\mathbf{|\mathfrak{G}|}$ & Compact?\\ \hline
Cyclic - $\dC$ &      $\mathbb{Z}/(n+1)$                                           & $SO(2)$   & Y                                 \\ \hline
Dihedral - $\dD$&  $D_{n+1}$                                                   & $O(2)$             & Y                        \\ \hline
Paracyclic - $\dC_\infty$&    $\mathbb{Z}$                                     & $\widetilde{SO(2)}$ & N                           \\ \hline
Paradihedral - $\dD_\infty$&    $D_\infty$                                          & $\widetilde{O(2)}$  & N                           \\ \hline
$N$-Cyclic  - $\dC_N$&      $\mathbb{Z}/N(n+1)$                                      & $\text{Spin}_N(2)$ &Y                        \\ \hline
$N$-Dihedral  - $\dD_N$&       $D_{N(n+1)}$                                            & $\text{Pin}^+_N(2)$  & Y                      \\ \hline
$M$-Quaternion  - $\dQ_M$&        $Q_{M(n+1)}$                                        & $\text{Pin}^-_{2M}(2)$ & Y                     \\ \hline
\end{tabular}
\vspace{5mm}
\caption{The classification of planar crossed simplicial groups.}\label{cstable}
\end{table}
\end{definition}

The next result is one that will be crucial for the rest of the paper, and allows us to move from planar crossed simplicial group sets to topological spaces with group action.  It follows from Theorem \ref{real} with the properties of planar crossed simplicial groups taken into consideration.  A proof of the result in the cyclic case can be found at \cite[Proposition 2.8]{homotopycc}.

\begin{proposition}\label{realsing}
Let $\dG$ be a planar crossed simplicial group, then there is a \textit{realisation functor}
$$|-|_{\mathfrak{G}} \colon \dG\textbf{-Set} \to \textbf{Top}^{|\mathfrak{G}|}$$
such that the following diagram commutes:
$$\xymatrix{& \textbf{Top}^{|\mathfrak{G}|} \ar[d]^u \\
\dG\textbf{-Set} \ar[ur]^{|-|_{\mathfrak{G}}} \ar[r]_{|i^\ast -|} & \textbf{Top}}$$
Where $|i^\ast-|$ is the realisation functor on the underlying simplicial set, and $u$ is the forgetful functor.  This functor admits a right adjoint called the \textit{singular functor} $S_{\mathfrak{G}}$ defined as
$$S_{\mathfrak{G}}(X)_n = \text{Hom}_{\textbf{Top}^{|\mathfrak{G}|}}(|\mathfrak{G}| \times \Delta_n,X)$$
\end{proposition}

For the rest of this paper we will only be interested in the planar crossed simplicial groups which have a compact group as its realisation (the reason being that equivariant homotopy theory is concerned with actions of compact Lie groups). 

\section{Three Quillen Model Structures for Compact Planar Crossed Simplicial Group Sets}\label{models}

In this section we will be interested in model structures in the sense of Quillen \cite{quillen}. We develop three  model structures for each compact planar crossed simplicial group.  These models arise from the following theorem which is one of the basis for equivariant homotopy theory (\cite{elmendorf, equivariantmay}):

\begin{theorem}[{\cite[Appendix A, Proposition 2.10]{MR1629858}}]
Let $G$ be a compact Lie group, and $\textbf{Top}^G$ the category of spaces with $G$ action.  Let $\mathscr{F}$ be any family of closed subsets of $G$, then there is a model structure on $\textbf{Top}^G$ in which a map $f \colon X \to Y$ is a:
\begin{enumerate}
\item Weak equivalence if the induced maps $f^H \colon X^H \to Y^H$ are weak equivalences for all $H \in \mathscr{F}$.
\item Fibration if the induced maps $f^H \colon X^H \to Y^H$ are fibrations for all $H \in \mathscr{F}$.
\item Cofibration if it has the left-lifting property with respect to trivial fibrations.
\end{enumerate}
\end{theorem}

We will be interested in developing model structures on planar crossed simplicial group sets which are Quillen equivalent to the above model, where we allow $\mathscr{F}$ to be one of the three following families:

\begin{itemize}
\item $\mathscr{F} = \emptyset$ denoted $\textbf{Top}^{> G}$ called the \textit{weak model}.
\item $\mathscr{F} = \{\text{all closed proper subgroups of } G\}$ denoted $\textbf{Top}^{< G}$ called the \textit{strong model}.
\item $\mathscr{F} = \{\text{all closed subgroups of } G\}$ denoted $\textbf{Top}^{\leq G}$ called the \textit{coupled model} (for reasons explained later).
\end{itemize}

These model structures are known for $\dC$, $\dD$ and $\dQ$, the results of this section are showing how simply we can lift these models to the $N$-fold cases.

\begin{remark}
There is also a collection of intermediate model structures which occur between the weak and strong model structures, however we will not consider them here as they do not add to the theory.
\end{remark}

\subsection{The Weak Model}

The first model that we will look at is the weak model as introduced by Dwyer, Hopkins and Kan \cite{homotopycc}. Note that this model structure works for any crossed simplicial group, not just compact planar ones. 

\begin{proposition}[{\cite[Theorem 6.2]{homotopycc}}]
The category $\dG\textbf{-Set}$ has a cofibrantly generated model structure where a map $f: X \to Y$ is a:
\begin{enumerate}
\item Weak equivalence if $i^\ast \left(f\right) \colon i^\ast \left(X\right)\to i^\ast \left(Y\right)$ is a weak equivalence in $\textbf{sSet}_\text{Quillen}$.
\item Fibration if $i^\ast \left(f\right) \colon i^\ast \left(X\right)\to i^\ast \left(Y\right)$ is a fibration in $\textbf{sSet}_\text{Quillen}$.
\item Cofibration if it has the left-lifting property with respect to trivial fibrations.
\end{enumerate}
The generating cofibrations are given by $\{\partial  \dG[n] \hookrightarrow \dG[n]\}$ (i.e., boundary inclusions defined in the obvious way).  We will call this model the \textit{weak model} and will denote it $\dG\textbf{-Set}^w$.
\end{proposition}

\begin{proposition}
Let $\dG$ be a compact planar crossed simplicial group.  Then there is a Quillen equivalence
$$\dG\textbf{-Set}^w \rightleftarrows \textbf{Top}^{> |\mathfrak{G}|}$$
with the equivalence given using the realisation and singular functors of Proposition \ref{realsing}.
\end{proposition}

\begin{proof}
The cyclic case is \cite[Theorem 4.1]{homotopycc}.  The proof hinges on the properties of the realisation/singular adjunction, and it can be checked that all of the functors in Proposition \ref{realsing} satisfy the required properties.  Therefore the other cases follows \textit{mutatis mutandis}.
\end{proof}

\subsection{The Strong Model}

We now look at the strong model structure which was introduced by Spali\'{n}ski for $\dC$ , $\dD$ and $\dQ$ (\cite{spalinskistrong, Spaliński2000557}).  We will begin by reviewing the cyclic case, extending it to the $N$-fold cyclic category.  The $N$-fold cases follow almost immediately from the $N=1$ case, and could be considered a corollary of the work of Spali\'{n}ski.  In particular, Spali\'{n}ski proved the following theorem (see \cite[\S 2]{spalinskistrong} for details on specific terminology in the assumptions).

\begin{theorem}[{\cite[Theorem 2.14]{spalinskistrong}}]\label{assumptionsformodel}
Let $\mathscr{D}$ be a category closed under coproducts, $I$ and arbitrary indexing set, and $\mathscr{H} = \{(\Psi_i,\Phi_i) \mid i \in I\}$ a family of adjoint functors $\Psi_i : \textbf{sSet} \rightleftarrows \mathscr{D} : \Phi_i$ such that:
\begin{enumerate}
\item $\mathscr{D}$ has all finite limits and arbitrary small colimits.
\item For all horn inclusions $f$, all $i,j \in I$, $\Phi_i \Psi_j(f)$ is a trivial cofibration.
\item For $X$ a simplicial horn or boundary, and $j \in I$, the object $\Psi_iX$ is $\Psi_\ast$-sequentially small with respect to all horn and boundary inclusions.
\item For all $i \in I$ the functor $\Phi_i$:
	\begin{enumerate}
		\item Preserves coproducts.
		\item Takes $\Psi_\ast$ regular pushouts to homotopy pushout diagrams.
		\item Preserves sequential colimits in $\mathscr{D}$ in which the morphisms are $\Psi_\ast$ induced from horn and boundary inclusions.
	\end{enumerate}
\end{enumerate}
Then there is a model structure on $\mathscr{D}$ in which a functor $f \colon X \to Y$ is a:
\begin{enumerate}
\item Weak equivalence if and only if for all $i \in I$, the map $\Phi_{i}(f) \colon \Phi_{i}(X) \to \Phi_{i}(Y)$ is a weak equivalence in $\textbf{sSet}_\text{Quillen}$.
\item Fibration if and only if for all $i \in I$, the map $\Phi_{i}(f) \colon \Phi_{i}(X) \to \Phi_{i}(Y)$ is a fibration in $\textbf{sSet}_\text{Quillen}$.
\item Cofibration if it has the left-lifting property with respect to trivial fibrations.
\end{enumerate}
\end{theorem}

To use this theorem for model structures on $\dG \textbf{-Set}$ we need only define the family of adjoint functors.  We will only define one side of the adjunction, namely the $\Phi_i$ family, as the $\Psi_i$ are simply constructed from these in a slightly technical way.  We begin with the cyclic case.

\begin{definition}
Let $X$ be a simplicial set, define the \textit{$r$-fold edgewise subdivision}  of $X$ to be the simplicial set $\text{sd}_rX$ where $\text{sd}_rX_n = X_{r(n+1)-1}$, and the face and degeneracy maps are described in terms of the face and degeneracies of $X$ in the following way:
\begin{align*}
d_i' &= d_i \cdot d_{i+(n+1)} \cdot \cdots \cdot d_{i+(r-1)(n+1)} \colon X_{r(n+1)-1} \to X_{rn-1},\\
 s_i' &= s_{i+(r-1)(n+1)} \cdot \cdots \cdot s_{i+(n+1)} \cdot s_i \colon  X_{rn-1} \to X_{r(n+1)-1}.
\end{align*}
 \end{definition}
 
 \begin{figure}[ht!]
\centering
\begin{tikzpicture}[scale = 0.35]
\begin{scope}[thick,decoration={
    markings,
    mark=at position 0.5 with {\arrow{>}}}
    ] 
\draw[postaction={decorate}]  (5,0) --(0,0);
\draw[postaction={decorate}]  (10,0) --(5,0);
\draw[postaction={decorate}]  (2.5,4.33) -- (0,0);
\draw[postaction={decorate}]  (5,8.66) -- (2.5,4.33);
\draw[postaction={decorate}]  (5,8.66) -- (7.5,4.33);
\draw[postaction={decorate}]  (7.5,4.33) -- (10,0);

\draw[postaction={decorate}] (2.5,4.33) -- (5,0);
\draw[postaction={decorate}] (7.5,4.33) -- (5,0);
\draw[postaction={decorate}] (7.5,4.33) -- (2.5,4.33);
\end{scope}
\end{tikzpicture}
\caption{2-Fold Edgewise Subdivision of $\Delta[2]$.}
\end{figure}
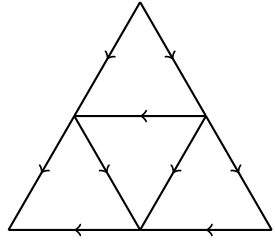

 \begin{remark}\label{cycaction}
 Let $X \in \dC_N\textbf{-Set}$ be an $N$-fold cyclic set, let $\text{sd}_{Nr}X$ denote the subdivision of the underlying simplicial set.  Then $\text{sd}_{Nr}X$ carries a natural $\mathbb{Z}/{Nr}$ action given by
\begin{align*}
\theta_n \colon \text{sd}_{Nr}X_n &\to \text{sd}_{Nr}X_n\\
x &\mapsto \tau_{Nr(n+1)-1}^{n+1}(x)
\end{align*}

 where $\tau_n \colon [n] \to [n]$ is the cyclic operator.  Therefore for every $r \geq 1$ there is a functor
 \begin{align*}
 \Phi_{Nr} \colon \dC_N\textbf{-Set} &\to \textbf{sSet}\\
 X &\mapsto (\text{sd}_{Nr}X)^{\mathbb{Z}/Nr}
 \end{align*}
 \end{remark}

\begin{proposition}\label{cycNmodel}
The category $\dC_N\textbf{-Set}$ of $N$-fold cyclic sets has a cofibrantly generated model category structure in which a map $f \colon X \to Y$ is a:
\begin{enumerate}
\item Weak equivalence if and only if for all $r \geq 1$, the map $\Phi_{Nr}(f) \colon \Phi_{Nr}(X) \to \Phi_{Nr}(Y)$ is a weak equivalence in $\textbf{sSet}_\text{Quillen}$.
\item Fibration if and only if for all $r \geq 1$, the map $\Phi_{Nr}(f) \colon \Phi_{Nr}(X) \to \Phi_{Nr}(Y)$ is a fibration in $\textbf{sSet}_\text{Quillen}$.
\item Cofibration if it has the left-lifting property with respect to trivial fibrations.
\end{enumerate}
We call the above model structure on $\dC_N\textbf{-Set}$ the \textit{strong model structure} and we will denote it by $\dC_N\textbf{-Set}^s$.  Moreover there is a Quillen equivalence $\dC_N\textbf{-Set}^s \rightleftarrows \textbf{Top}^{< Spin_N(2)}$.
\end{proposition}

\begin{proof}
The case of $N = 1$ is given in \cite[Theorem 3.10]{spalinskistrong}, where it is shown that the functors $\Phi_{r}$ are part of an adjunction which satisfy the conditions of Theorem \ref{assumptionsformodel}.  For $N > 1$, we are looking at a sub-collection of the functors for $N=1$, which therefore still have the desired properties. Therefore we have the existence of the model structure.  For the Quillen adjunction, the detailed case of $N = 1$ is given in \cite[Theorem 5.1]{spalinskistrong}, where it is shown using the cyclic realisation/singular adjunction.  As will be the norm with this paper, due to the properties shared by the cyclic and $N$-fold cyclic realisations, these conditions trivially hold for the $N$-fold case, and the result follows.
\end{proof}

We now move on to consider the dihedral and quaternionic cases.  In this case we need to be able to access the fixed point data of the reflexive part of the groups, so to do this we introduce another subdivision functor.

\begin{definition}
Let $X$ be a simplicial set, define the \textit{Segal subdivision}  of $X$ to be the simplicial set $\text{sq}X$ where $\text{sq}X_n = X_{2n+1}$, and the face and degeneracy maps are described in terms of the face and degeneracies of $X$ in the following way:
\begin{align*}
d_i' &= d_i \cdot d_{2n + 1 - i} \colon X_{2n+1} \to X_{2n-1},\\
 s_i' &= s_{2n-1} \cdot s_i \colon X_{2n-1} \to X_{2n+1}.
 \end{align*}
\end{definition} 

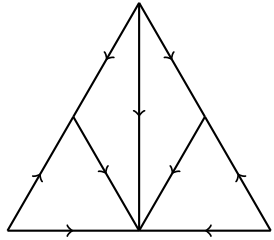
\begin{figure}[ht!]
\centering
\begin{tikzpicture}[scale = 0.35]
\begin{scope}[thick,decoration={
    markings,
    mark=at position 0.5 with {\arrow{>}}}
    ] 
\draw[postaction={decorate}]  (0,0) --(5,0);
\draw[postaction={decorate}]  (10,0) --(5,0);
\draw[postaction={decorate}]  (0,0) -- (2.5,4.33);
\draw[postaction={decorate}]  (5,8.66) -- (2.5,4.33);
\draw[postaction={decorate}]  (5,8.66) -- (7.5,4.33);
\draw[postaction={decorate}]  (10,0) -- (7.5,4.33);

\draw[postaction={decorate}] (2.5,4.33) -- (5,0);
\draw[postaction={decorate}] (7.5,4.33) -- (5,0);
\draw[postaction={decorate}] (5,8.66) -- (5,0);
\end{scope}
\end{tikzpicture}
\caption{Segal Subdivision of $\Delta[2]$.}
\end{figure}

 \begin{remark}\label{dihaction}
 If $X \in \dD_N\textbf{-Set}$ is an $N$-fold dihedral set, let $\text{sq}X$ denote the Segal subdivision of the underlying simplicial set.  Then $\text{sq}X$ carries a natural $\mathbb{Z}/2$ action given by
 \begin{align*}
 \rho_n \colon \text{sq}X &\to \text{sq}X\\
 x &\mapsto \omega_{2n+2}(x)
 \end{align*}
 where $\omega_n \colon [n] \to [n]$ is the reflexive operator. Therefore there is a functor
 \begin{align*}
 \Gamma_1 \colon \dD\textbf{-Set} &\to \textbf{sSet}\\
 X &\mapsto (\text{sq}X)^{\mathbb{Z}/2}
 \end{align*}
 \end{remark}

Using a combination of the two subdivision functors that we have introduced we can now access the fixed point data for any dihedral group.

\begin{definition}
For $X$ a simplicial set and $r \geq 1$, define the \textit{$r$-dihedral subdivision functor} $\text{sbd}_r$ to be
$$\text{sbd}_rX = \text{sq}(\text{sd}_r(X))$$
Which by Remarks \ref{cycaction} and \ref{dihaction} gives a functor
\begin{align*}
\Gamma_{Nr} \colon \dD_N\textbf{-Set} &\to \textbf{sSet}\\
X &\mapsto (\text{sbd}_rX)^{D_Nr}
 \end{align*}
\end{definition}

\begin{proposition}
The category $ \dD_N\textbf{-Set}$ has a  cofibrantly generated model structure where a map $f \colon X \to Y$ is a:
\begin{enumerate}
\item Weak equivalence if for all $r \geq 1$ the maps $\Phi_{Nr}(f) \colon \Phi_{Nr}(X) \to \Phi_{Nr}(Y), \Gamma_{Nr}(f) \colon \Gamma_{Nr}(X) \to \Gamma_{Nr}(Y)$ are weak equivalences in $\textbf{sSet}_\text{Quillen}$.
\item Fibrations if for all $r \geq 1$ the maps $\Phi_{Nr}(f) \colon \Phi_{Nr}(X) \to \Phi_{Nr}(Y), \Gamma_{Nr}(f) \colon \Gamma_{Nr}(X) \to \Gamma_{Nr}(Y)$ are fibrations in $\textbf{sSet}_\text{Quillen}$.
\item Cofibration if it has the left-lifting property with respect to the class of trivial fibrations.
\end{enumerate}
We call the above model structure on $\dD_N\textbf{-Set}$ the \textit{strong model structure} and we will denote it by $\dD_N\textbf{-Set}^s$.  Moreover there is a Quillen equivalence $\dD_N\textbf{-Set}^s \rightleftarrows \textbf{Top}^{< Pin^+_N(2)}$.
\end{proposition}

\begin{proof}
The existence of the model for $N=1$ is given in \cite[Theorem 4.3]{Spaliński2000557}, as we already have the functors $\Phi_{r}$ satisfying the conditions of Theorem \ref{assumptionsformodel}, what is shown is that the family $\Gamma_{r}$ also satisfies these assumptions.  The Quillen equivalence for $N=1$ is given in \cite[Theorem 4.3]{Spaliński2000557}.  The case for $N>1$ follows as discussed in the proof of Proposition \ref{cycNmodel}.
\end{proof}

We complete this section with a model for $\dQ_M\textbf{-Set}$ which follows very easily from the dihedral case.  We see that the group $Q_r$ also acts on $\text{sbd}_rX$ where $X$ is a quaternionic set.  Therefore for $r \geq 1$ we define the functor
\begin{align*}
\Gamma_r^q \colon \dQ\textbf{-Set} &\to \textbf{sSet}\\
 X &\mapsto (\text{sbd}_rX)^{Q_r}
\end{align*}

\begin{lemma}[{\cite[Theorem 5.2, Theorem 5.3]{Spaliński2000557}}]
The category $ \dQ_M\textbf{-Set}$ has a cofibrantly generated model structure where a map $f \colon X \to Y$ is a:
\begin{enumerate}
\item Weak equivalence if for all $r \geq 1$ the maps $\Phi_{Mr}(f) \colon \Phi_{Mr}(X) \to \Phi_{Mr}(Y), \Gamma^q_{Mr}(f) \colon \Gamma^q_{Mr}(X) \to \Gamma^q_{Mr}(Y)$ are weak equivalences in $\textbf{sSet}_\text{Quillen}$.
\item Fibrations if for all $r \geq 1$ the maps $\Phi_{Mr}(f) \colon \Phi_{Mr}(X) \to \Phi_{Mr}(Y), \Gamma^q_{Mr}(f) \colon \Gamma^q_{Mr}(X) \to \Gamma^q_{Mr}(Y)$ are fibrations in $\textbf{sSet}_\text{Quillen}$.
\item Cofibration if it has the left-lifting property with respect to the class of trivial fibrations.
\end{enumerate}
We call the above model structure on $\dQ_M\textbf{-Set}$ the \textit{strong model structure} and we will denote it by $\dQ_M\textbf{-Set}^s$.  Moreover there is a Quillen equivalence $\dQ_M\textbf{-Set}^s \rightleftarrows \textbf{Top}^{< Pin^-_{2M}(2)}$.
\end{lemma}

\subsection{The Coupled Model}

We now move on to the last collection of model categories for $\dG \textbf{-Set}$.  In this case we will develop discrete models which respect all closed subgroups of $G$, not just the proper subgroups.  Note that  the action, for example, of $SO(2)$ on any cyclic set is discrete. In \cite{MR1295162} the fixed point set is given as
$$\text{Fix} = \{ x \in X_0 \mid s_0x = s_1x = t_1s_0x\}.$$
Therefore to capture this extra data, we must construct a new category which extends $\dG\textbf{-Set}$ in such a way we can keep track of the infinite group action.

\begin{definition}
Let $\mathscr{C}$ and $\mathscr{D}$ be categories, and $F \colon \mathscr{C} \to \mathscr{D}$ a functor.  Then the category $\mathscr{C}_F\mathscr{D}$ has:
\begin{enumerate}
\item Objects given by triples $\left(A,B,FA \to B\right)$ with $A \in \mathscr{C}$ and $B \in \mathscr{D}$.
\item Morphisms specified by maps $f_1 \colon A \to A'$ and $f_2 \colon B \to B'$ such that the following diagram commutes:
$$\xymatrix{FA \ar[r] \ar[d]_{Ff_1} & B \ar[d]^{f_2} \\
FA' \ar[r] & B'}$$
\end{enumerate}

We will call such a category \emph{coupled}.
\end{definition}

\begin{remark}
The above category construction can be seen as the comma category $\left(f/1_D\right)$ \cite{lane1998categories}.  Therefore it can also be seen as a presentation of the lax limit of $F \colon A \to B$ in $\mathfrak{Cat}$ \cite{laxlimit}.
\end{remark}

\begin{definition}
Let $\mathscr{C}$ and $\mathscr{D}$ be model categories.  A functor $F \colon \mathscr{C} \to \mathscr{D}$ is \textit{Reedy admissible} if $F$ preserves colimits and $F$ has the property that given a morphism $\left(A,B,FA \to B\right) \to \left(A',B',FA' \to B'\right)$ in $\mathscr{C}_F\mathscr{D}$ such that $A' \to A$ is a trivial cofibration in $\mathscr{C}$ and $FA' \bigcup_{FA} B \to B'$ is a trivial cofibration in $\mathscr{C}$ then $B \to B'$ is a weak equivalence in $\mathscr{D}$.  In particular, left Quillen functors are Reedy admissible.
\end{definition}

\begin{theorem}[{\cite[Theorem 1.1]{blumberg}}]\label{coupledmodeltheorem}
Let $\mathscr{C}$ and $\mathscr{D}$ be model categories and $F \colon \mathscr{C} \to \mathscr{D}$ a Reedy admissible functor.  Then $\mathscr{C}_F\mathscr{D}$ admits a model structure where a map $\left(A,B,FA \to B\right) \to \left(A',B',FA' \to B'\right)$ in $\mathscr{C}_F\mathscr{D}$ is a:
\begin{enumerate}
\item Weak equivalence if $A \to A'$ is a weak equivalence in $\mathscr{C}$ and $B \to B'$ is a weak equivalence in $\mathscr{D}$.
\item Fibration if $A \to A'$ is  a fibration in $\mathscr{C}$ and $B \to B'$ is a fibration in $\mathscr{D}$.
\item Cofibration if $A \to A'$ is a cofibration in $\mathscr{C}$ and $FA' \bigcup_{FA}B \to B'$ is a cofibration in $\mathscr{D}$.
\end{enumerate}
\end{theorem}

\begin{lemma}
If $A$, $B$ are cofibrantly generated model categories and $F: A \to B$ is Reedy admissible, then the category $A_FB$ is cofibrantly generated.
\end{lemma}

\begin{proof}
We know that a map in $\mathscr{A}_F\mathscr{B}$ is a cofibration if and only if $A \to A'$ is a cofibration in $\mathscr{A}$ and $FA' \bigcup_{FA}B \to B'$ is a cofibration in $\mathscr{B}$.  We also know that $\mathscr{A}$ and $\mathscr{B}$ have generating sets of cofibrations.  Therefore the generating cofibrations of $\mathscr{A}_F \mathscr{B}$ is some $C \subset A \times B$.
\end{proof}

We begin with the cyclic case as proved by Blumberg \cite{blumberg}, to do this we construct a functor $\nabla \colon \textbf{sSet} \to \dC \textbf{-Set}^s$ such that $|X|^H \simeq |\nabla X|^H_{\mathfrak{C}}$ for all finite $H < SO(2)$.

\begin{definition}
Let $\nabla_n = S_{\mathfrak{C}}\left(|\Delta[n]|\right)$.  Then we have that $\nabla_\ast$ is a cosimplicial cyclic set.  We then define
\begin{align*}
\nabla \colon \textbf{sSet} &\to \dC \textbf{-Set}^s\\
\nabla X &= X \otimes_{\Delta^{op}} \nabla_\ast
\end{align*}
This functor has a right adjoint $A$ given by $A\left(Y\right)_n = \text{Hom}_{\dC \textbf{-Set}}\left(\nabla_n,Y\right)$
\end{definition}

We consider the category $\mathscr{P} := \textbf{sSet}_\nabla \dC \textbf{-Set}^s$.

\begin{lemma}[{\cite[Lemma 4.2]{blumberg}}]\label{commute}
The functor $\left(-\right)^G$ on based $G$-spaces preserves pushouts of diagrams where one leg is a closed inclusion.
\end{lemma}

\begin{lemma}[{\cite[Lemma 4.5]{blumberg}}]\label{proof:reedy}
The functor $\nabla$ is Reedy admissible.
\end{lemma}

\begin{proof}
To show this we need to show that given a map $\left(A,B, \nabla A \to B\right) \to \left(A', B', \nabla A' \to B'\right)$ in $\mathscr{P}$ such that $A \to A'$ is a trivial cofibration and $\nabla A' \bigcup_{\nabla A} B \to B'$ is a trivial cofibration, then the map $B \to B'$ is a weak equivalence.  We begin by noting that the map $B \to B'$ is the composite
$$B \to \nabla A' \bigcup_{\nabla A} B \to B'$$
and we know by assumption that the second map is a weak equivalence.  Therefore it will suffice just to show that the first map is a weak equivalence.  By taking the realisation this is equivalent to showing that
$$|B|_{\mathfrak{C}}^H \to |\nabla A' \bigcup_{\nabla A} B|^H_{\mathfrak{C}}$$
is a weak equivalence of spaces for all finite $H \subset SO(2)$.  We then use the fact that the geometric realisation is a colimit to see that
$$|\nabla A' \bigcup_{\nabla A} B|^H_{\mathfrak{C}} \simeq \left(|\nabla A'|_{\mathfrak{C}} \bigcup_{|\nabla A|_{\mathfrak{C}}} |B|_{\mathfrak{C}}\right)^H$$
Since $A \to A'$ is a cofibration of simplicial sets, it is an inclusion and $|\nabla A|_{\mathfrak{C}} \to |\nabla A'|_{\mathfrak{C}}$ is a closed inclusion. By Lemma \ref{commute} the fixed-point functor commutes with the pushout so we get
$$\left(|\nabla A'|_{\mathfrak{C}} \bigcup_{|\nabla A|_{\mathfrak{C}}} |B|_{\mathfrak{C}}\right)^H \simeq |\nabla A'|_{\mathfrak{C}}^H \bigcup_{|\nabla A|_{\mathfrak{C}}^H} |B|_{\mathfrak{C}}^H$$
Therefore we see that $|B|_{\mathfrak{C}}^H \to |\nabla A'|_{\mathfrak{C}}^H \bigcup_{|\nabla A|_{\mathfrak{C}}^H} |B|_{\mathfrak{C}}^H$ is a pushout of a trivial cofibration and therefore itself a trivial cofibration.
\end{proof}

\begin{corollary}[{\cite[Corollary 4.7]{blumberg}}]\label{coupledmodel}There exists a cofibrantly generated model structure on $\mathscr{P}$ in which a map is a:
\begin{enumerate}
\item Weak equivalence if $A \to A'$ is a weak equivalence of simplicial sets and $B \to B'$ is a strong weak equivalence of cyclic sets.
\item Fibration if $A \to A'$ is a fibration of simplicial sets and $B \to B'$ is a strong fibration of cyclic sets.
\item Cofibration if $A \to A'$ is a cofibration of simplicial sets and $\nabla A' \bigcup_{\nabla A} B \to B'$ is a strong cofibration of cyclic sets.
\end{enumerate}
\end{corollary}

We now must construct a bijection between our two categories of interest, $\mathscr{P}$ and $\textbf{Top}^{SO(2)}$.

\begin{definition}
The functor $L \colon \mathscr{P} \to \textbf{Top}^{SO(2)}$ takes a triple $\left(A,B, \nabla A \to B\right)$ to the pushout in the diagram
$$\xymatrix{|\nabla A|_{\mathfrak{C}} \ar[r] \ar[d]_{\zeta} & |B|_{\mathfrak{C}} \ar[d] \\
|A| \ar[r] & X}$$
where the map $\zeta \colon |\nabla A|_{\mathfrak{C}} \to |A|$ is the natural map which induces weak equivalences on passage to fixed point subspaces for all finite subgroups of $SO(2)$. A morphism $\left(A,B,\nabla A \to B\right) \to \left(A',B',\nabla A' \to B'\right)$ gives rise to a commutative diagram:
$$\xymatrix{|A| \ar[d] & |\nabla A|_{\mathfrak{C}} \ar[r] \ar[l] \ar[d]& |B|_{\mathfrak{C}} \ar[d] \\
|A'| & |\nabla A'|_{\mathfrak{C}} \ar[l] \ar[r] & |B'|_{\mathfrak{C}}}$$
\end{definition}

\begin{definition}
The functor $R \colon \textbf{Top}^{SO(2)} \to \mathscr{P}$ takes $X$ to the triple $$\left(S\left(X^{SO(2)}\right),S_\mathfrak{C}\left(X\right),\nabla S\left(X^{SO(2)}\right) \to S_\mathfrak{C}\left(X\right)\right)$$
Here the map $\nabla S\left(X^{SO(2)}\right) \to S_\mathfrak{C}\left(X\right)$ is the adjoint to the composite
$$|\nabla S\left(X^{SO(2)}\right)|_{\mathfrak{C}} \to |S\left(X^{SO(2)}\right)| \to X^{SO(2)} \hookrightarrow X$$
A map $X \to Y$ in $\textbf{Top}^{SO(2)}$ induces maps $S\left(X^{SO(2)}\right) \to S\left(Y^{SO(2)}\right)$ and $S_\mathfrak{C}\left(X\right) \to S_\mathfrak{C}\left(Y\right)$ which fit in the following commutative diagram:
$$\xymatrix{\nabla S\left(X^{SO(2)}\right) \ar[r] \ar[d] & S_\mathfrak{C}\left(X\right) \ar[d] \\
\nabla S\left(Y^{SO(2)}\right) \ar[r] & S_\mathfrak{C}\left(Y\right)}$$
\end{definition}

\begin{theorem}[{\cite[\S 6]{blumberg}}]
The functors $L$ and $R$ specify a Quillen equivalence between $\mathscr{P}$ equipped with the model structure given in Corollary \ref{coupledmodel} and $\textbf{Top}^{\leq SO(2)}$.
\end{theorem}

For model structures of this type, we will frequently abuse notation and just denote
$$\dC\textbf{-Set}^c := \mathscr{P} := \textbf{sSet}_\nabla \dC \textbf{-Set}^s$$
and call it the \textit{coupled model} on cyclic sets.\label{index:coupled}  However we should remember that in fact $\dC\textbf{-Set}^c$ has a different underlying category than $\dC\textbf{-Set}^w$ and $\dC\textbf{-Set}^s$.

We can now consider changing the above machinery to work in  $\dC_N\textbf{-Set}$.  All of the following statements follow almost instantly from the machinery introduced by Blumberg which we reproduced above, and by using the strong model structures that we defined previously.  This really is one of the advantages of working with planar crossed simplicial groups, the properties that they share usually allow us to prove things about all of them after it has been asserted in a single case.  First of all we must think of what is the correct underlying category to use.  We will be considering the category  $\mathscr{P}_N := \textbf{sSet}_{\nabla^N}\dC_N\textbf{-Set}^s$ where
\begin{align*}
\nabla^N_n &= S_{\mathfrak{C}_N}(|\Delta[n]|),\\
\nabla^N X &= X\otimes_{\Delta^{op}} \nabla_\ast^N.
\end{align*}

\begin{lemma}
$\nabla^N$ is Reedy admissible $\forall N \in \mathbb{N}$.  Therefore there is a model structure on $\mathscr{P}_N$ given by Theorem \ref{coupledmodeltheorem}.
\end{lemma}

\begin{proof}
If we look at the proof of Lemma \ref{proof:reedy}, we see that the only tool that we required was the fact that the geometric realisation was a colimit combined with Lemma \ref{commute} which works for all $G$.  We have that the realisation $|-|_{\mathfrak{C}_N}$ is a colimit by definition, and therefore the proof is identical.
\end{proof}

\begin{proposition}
There is a Quillen equivalence
$$ L : \mathscr{P}_N \rightleftarrows \textbf{Top}^{\leq \text{Spin}_N\left(2\right)} : R$$
where the adjunctions $L$ and $R$ are the obvious ones obtained form the cyclic case.
\end{proposition}

\begin{proof}
As above, we can prove this statement using the proofs from the $N=1$ case.  The properties of $\dC$-sets that are used to prove the statement have been proved in the case of $\dC_N$-sets. 
\end{proof}

Again, we will ease notation and denote $\dC_N\textbf{-Set}^c := \mathscr{P}_N := \textbf{sSet}_{\nabla^N} \dC_N \textbf{-Set}^s$ and call it the \textit{coupled model} on $N$-fold cyclic sets.

The question now is how can we extend this to other planar crossed simplicial groups.  Spali\'{n}ski answered this question with respect to $\dD\textbf{-Set}$ in relation to $\textbf{Top}^{O(2)}$ in \cite{discrete02}.  We outline the adjustments needed from the above to facilitate this. In the case of $O(2)$, we have two infinite subgroups, namely $SO(2)$ and $O(2)$ itself.  Therefore we need some way to take into account the $\mathbb{Z}/2$ action which relates $SO(2)$ and $O(2)$.  Denote by $\Delta \mathfrak{R}$ the crossed simplicial group where $\mathfrak{G}_n = \mathbb{Z}/2$ for all $n$.  We can put a model structure on the category $\dR \textbf{-Set}$ which has a Quillen adjunction to the equivariant model structure on $\textbf{Top}^{\mathbb{Z}/2}$.  This model structure is discussed in \cite[\S 3]{discrete02}.

\begin{definition}
Let $\Xi_n = S_{\mathfrak{D}}\left(\Delta^n_{\mathbb{Z}/2}\right)$ where $\Delta^n_{\mathbb{Z}/2} := \mathbb{Z}/2 \times \Delta^n$.  Then we have that $\Xi_\ast$ is a cosimplicial dihedral set.  We then define
\begin{align*}
\Xi \colon  \dR \textbf{-Set} &\to \dD \textbf{-Set}^s,\\
\Xi X &= X \otimes_{\dR^{op}} \Xi_\ast.
\end{align*}
\end{definition}

We consider the category $\mathscr{K} := \dR \textbf{-Set}_\Xi \dD \textbf{-Set}^s$.  We see that this is related to the previous case, all we have done is replaced $\textbf{sSet}$ with $\dR \textbf{-Set}$ which will track the $\mathbb{Z}/2$ action.  As Spali\'{n}ski points out in \cite{discrete02}, the proofs involving the category $\mathscr{K}$ are analogous to the case proved by Blumberg due to the properties of the realisation functor.

\begin{lemma}[{\cite[Lemma 4.5]{discrete02}}]
The functor $\Xi \colon  \dR \textbf{-Set} \to \dD \textbf{-Set}^s$ is Reedy admissible.
\end{lemma}

\begin{corollary}\label{coupledmodel2}
There exists a model structure on $\mathscr{K}$ in which a map is a:
\begin{enumerate}
\item Weak equivalence if $A \to A'$ is a weak equivalence of $\dR$-sets and $B \to B'$ is a strong weak equivalence of dihedral sets.
\item Fibration if $A \to A'$ is a fibration of $\dR$-sets and $B \to B'$ is a strong fibration of dihedral sets.
\item Cofibration if $A \to A'$ is a cofibration of $\dR$-sets and $\Xi A' \bigcup_{\Xi A} B \to B'$ is a strong cofibration of dihedral sets.
\end{enumerate}
\end{corollary}

We now must construct a bijection between our two categories of interest, $\mathscr{K}$ and $\textbf{Top}^{O\left(2\right)}$.  We can do this in the akin to the cyclic case, and we get the following result.

\begin{theorem}[{\cite[\S 5]{discrete02}}]
There exists functors $L$ and $R$ which specify a Quillen equivalence between $\mathscr{K}$ equipped with the model structure given in Corollary \ref{coupledmodel2} and $\textbf{Top}^{\leq O\left(2\right)}$.
\end{theorem}

Now, again we modify to allow it to work for $\dD_N$, the proof of which is obtainable from the theory discussed above and using the ideas from the strong model structures.  We consider the category $\mathscr{K}_N := \dR\textbf{-Set}_{\Xi^N}\dD_N\textbf{-Set}^s$ where
\begin{align*}
\Xi^N_n &= S_{\mathfrak{D}_N}\left(\Delta^n_{\mathbb{Z}/2}\right),\\
\Xi^N X &= X\otimes_{\dR^{op}} \Xi_\ast^N.
\end{align*}

\begin{lemma}
$\Xi^N$ is Reedy admissible $\forall N \in \mathbb{N}$, therefore there is a model structure on $\mathscr{K}_N$ given by Theorem \ref{coupledmodeltheorem}.  Moreover there is a Quillen equivalence  $\mathscr{K}_N \rightleftarrows \textbf{Top}^{\leq \text{Pin}^+_N\left(2\right)}$.
\end{lemma}

Denote by $\Delta \mathfrak{D}_N\textbf{-Set}^c := \mathscr{K}_N := \dR\textbf{-Set}_{\Xi^N}\dD_N\textbf{-Set}^s$ the \textit{coupled model} on $N$-fold dihedral sets.

Finally, we define the $\dQ_M$ case, which properties follow after the obvious changes have been made.  We consider the category $\mathscr{Q}_M := \dR\textbf{-Set}_{\Psi^M}\dQ_M\textbf{-Set}^s$ where
\begin{align*}
\Psi^M_n &= S_{\mathfrak{Q}_M}\left(\Delta^n_{\mathbb{Z}/2}\right),\\
\Psi^M X &= X \otimes_{\dR^{op}} \Psi_\ast^M.
\end{align*}

\begin{lemma}
$\Psi^M$ is Reedy admissible $\forall M \in \mathbb{N}$, therefore there is a model structure on $\mathscr{Q}_M$ given by Theorem \ref{coupledmodeltheorem}.  Moreover there is a Quillen equivalence $\mathscr{Q}_M \rightleftarrows \textbf{Top}^{\leq \text{Pin}^-_{2M}\left(2\right)}$.
\end{lemma}

Denote by $\Delta \mathfrak{Q}_M\textbf{-Set}^c :=\mathscr{Q}_M := \dR\textbf{-Set}_{\Psi^M}\dQ_M\textbf{-Set}^s$ the \textit{coupled model} on $M$-fold quaternionic sets.

\section{Presheaves with Values in Compact Planar Crossed Simplicial Groups}\label{sec4}

In this section we will develop local model structures on presheaves with values in compact planar crossed simplicial groups with the three different models discussed in Section \ref{models}.  We will begin by recalling the theory of simplicial presheaves as discussed, at length, in \cite{quillenpre}.

\subsection{Simplicial Presheaves}

\begin{definition}
Let $\mathscr{C}$ be a simplicial category, a \textit{topology} on $\mathscr{C}$ is a Grothendieck topology on the category $\text{Ho}(\mathscr{C})$.  A \textit{simplicial site} $(\mathscr{C},\tau)$ is the data of $\mathscr{C}$ along with the topology $\tau$.
\end{definition}

\begin{definition}
Let $\mathscr{C}$ be a simplicial category, the category of \textit{simplicial presheaves} is the functor category
$$\textbf{sPr}(\mathscr{C}) = \textbf{Fun}(\mathscr{C}^{op},\textbf{sSet})$$
\end{definition}

We will now put the first model structure on the category $\textbf{sPr}(\mathscr{C})$, which will not use any properties of the topology of the simplicial site.  This is known as the point-wise (or sometimes global) projective structure and was originally developed in \cite{bousfield1972homotopy}.

\begin{proposition}
There exists a \textit{point-wise model structure} on the category $\textbf{sPr}(\mathscr{C})$ where a map $f:\mathscr{F} \to \mathscr{F}'$ is a:
\begin{enumerate}
\item Weak equivalence if for any $X \in \mathscr{C}$ the map $\mathscr{F}(X) \to \mathscr{F}'(X)$ is a weak equivalence of simplicial sets.
\item Fibration if for any $X \in \mathscr{C}$ the map $\mathscr{F}(X) \to \mathscr{F}'(X)$ is a fibration of simplicial sets.
\item Cofibration if it has the left-lifting property with respect to the class of trivial fibrations.
\end{enumerate}
\end{proposition}

To reflect the topology of the site, we will add more weak equivalences.  The following definition is one way to formulate these in a base-point free way.

\begin{definition}
A map $f \colon \mathscr{F} \to \mathscr{F}'$ in $\textbf{sPr}(\mathscr{C})$ is a \textit{local weak equivalence} if:
\begin{itemize}
\item the induced map $\widetilde{\pi_0}\mathscr{F} \to \widetilde{\pi_0}\mathscr{F}'$ is an isomorphism of sheaves, where $\widetilde{\pi_0}$ is the sheafification of $\pi_0$.
\item the induced squares
$$\xymatrix{\pi_n F \ar[r] \ar[d] &\ar[d] \pi_n F' \\
F_0 \ar[r] &  F'_0}$$
are pullbacks after sheafification.
\end{itemize}
\end{definition}

\begin{theorem}[{\cite[\S 3]{jardinesimplicialpresheaves}}]
Let $(\mathscr{C}, \tau)$ be a simplicial site.  There exists a cofibrantly generated \textit{local model structure} on the category $\textbf{sPr}(\mathscr{C})$ where a map $f:\mathscr{F} \to \mathscr{F}'$ is a:
\begin{enumerate}
\item Weak equivalence if it is a local weak equivalence
\item Fibration if it has the right-lifting property with respect to the trivial cofibrations
\item Cofibration if it is a cofibration in the point-wise model.
\end{enumerate}
We will denote this model structure $\textbf{sPr}_\tau(\mathscr{C})$.
\end{theorem}

\begin{definition}
A \textit{hypercover} of an object $X \in \mathscr{C}$ is the data of a simplicial presheaf $\mathscr{H}$ together with a morphism $\mathscr{H} \to X$ such that:
\begin{enumerate}
\item For any integer $n$, the presheaf $\mathscr{H}_n$ is a disjoint union of representable presheaves.
\item For any $n \geq 0$, the morphism of presheaves
$$\mathscr{H}_n \simeq \text{Hom}(\Delta^n,\mathscr{H}) \to \text{Hom}(\partial \Delta^n,\mathscr{H}) \times_{\text{Hom}(\partial \Delta^n,X)}\text{Hom}( \Delta^n,X)$$
induces an epimorphism after sheafification.
\end{enumerate}
\end{definition}

\begin{remark}\label{localise}
It can be shown that $\textbf{sPr}_\tau(\mathscr{C})$ is in fact the left Bousfield localisation of $\textbf{sPr}(\mathscr{C})$ with the point-wise model structure with respect to hypercovers.
\end{remark}

\subsection{Presheaves for the Weak and Strong Models}

We will begin by looking at presheaves over the weak and strong models.  This was developed in the cyclic case by Seteklev and {\O}stv{\ae}r in \cite{cyclicsheaves}, all results here can be obtained from the results presented there.

\begin{definition}
Let $\dG$ be a compact planar crossed simplicial group and $\mathscr{C}$ a small Grothendieck site.  A \textit{$\dG$-presheaf} on $\mathscr{C}$ is a functor $\mathcal{F} \colon \mathscr{C}^{op} \to \dG\textbf{-Set}$. We will denote by $\dG\textbf{-Pr}(\mathscr{C})$ the category of all $\dG$-presheaves on $\mathscr{C}$.
\end{definition}

From now on we will focus on equipping the category of $\dG$-presheaves with the strong model, however all of the results that follow also hold for the weak model.  Note that, as discussed in \cite{cyclicsheaves}, the category $\dG\textbf{-Pr}(\mathscr{C})$ is tensored, cotensored and enriched in $\dG\textbf{-Set}$.  As in the simplicial case, we can put a point-wise model structure on $ \dG\textbf{-Pr}(\mathscr{C})$ which does not look at all at the topology on $\mathscr{C}$, whose existence follows from the fact that $\dG \textbf{-Set}^s$ is cofibrantly generated.  For the cyclic case this is given as \cite[Theorem 4.2]{cyclicsheaves}.

\begin{proposition}
There is a cofibrantly generated model structure on $ \dG\textbf{-Pr}(\mathscr{C})$  called the \textit{projective point-wise model structure} where a map is a:
\begin{enumerate}
\item Weak equivalence if it is a point-wise weak equivalence in $\dG \textbf{-Set}^s$.
\item Fibration if it is a point-wise fibration in $\dG \textbf{-Set}^s$.
\item Cofibration if it has the left-lifting property with respect to the trivial fibrations.
\end{enumerate}
The generating set of cofibrations is given by $C \otimes i$ where $i$ is a generating cofibration of $\dG\textbf{-Set}^s$ and $C \in \mathscr{C}$.  We will denote this model $ \dG\textbf{-Pr}^s(\mathscr{C})$
\end{proposition}

We will now construct the local model structure on the $\dG$-presheaves.  To do this we will take the approach of Remark \ref{localise} and define the strong descent condition for hypercovers of cyclic presheaves, and then localise at this class of maps (see \cite[Theorem 5.3]{cyclicsheaves}).

\begin{definition}
A fibrant object $\mathscr{F}$ of $ \dG\textbf{-Pr}^s(\mathscr{C})$ with the point-wise model structure satisfies \textit{strong descent} if for every hypercover $\mathscr{H} \to X$ the induced map of $\dG$-sets
$$\mathscr{F}(X) \to\text{Holim}_{[n] \in \Delta}\mathscr{F}(\mathscr{H}_n)$$
is a weak equivalence in $\dG\textbf{-Set}^s$.
\end{definition}

\begin{definition}
There is a cofibrantly generated \textit{strong local model structure} on $ \dG\textbf{-Pr}^s(\mathscr{C})$ given as the left Bousfield localisation of the point-wise model at the class of hypercovers.  We will denote this model $ \dG\textbf{-Pr}_\tau(\mathscr{C})^s$
\end{definition}

We finish this section by outlining how we can explicitly describe the weak equivalences in these models. We will look at the construction for the strong cyclic case, with the other cases being the same with the correct replacement of subdivision functors as discussed in Section \ref{models}. First of all note that the functors $\Phi_r \colon \dC \textbf{-Set} \to \textbf{sSet}$ induce a family of functors $\Phi_{r \ast} \colon \dC \textbf{-Set}_\ast \to \textbf{sSet}_\ast$ between the pointed versions of these categories.  For $X$ a fibrant cyclic set and $x \in X_0$ we define $\pi^r_n(X,x) := \pi_n \Phi_{r \ast} (X,x)$.

\begin{definition}
A map $f \colon \mathscr{F} \to \mathscr{F}'$ in $\dC \textbf{-Pr}(\mathscr{C})$ is a \textit{strong local weak equivalence} if for all $r \geq 1$:
\begin{itemize}
\item the induced map $\widetilde{\pi^r_0}\mathscr{F} \to \widetilde{\pi^r_0}\mathscr{F}'$ is an isomorphism of sheaves.
\item the induced squares
$$\xymatrix{\pi^r_n F \ar[r] \ar[d] &\ar[d] \pi^r_n F' \\
\Phi_r F_0 \ar[r] & \Phi_r F'_0}$$
are pullbacks after sheafification.
\end{itemize}
\end{definition}

\begin{remark}
From the definition of $\pi^r_n(X,x)$ it is easy to see that a map $f \colon \mathscr{F} \to \mathscr{F}'$ in $\dC \textbf{-Pr}(\mathscr{C})$ is a strong local weak equivalence if and only if $\Phi_r(f)$ is a local weak equivalence of simplicial presheaves for all $r \geq 1$.  This means that we could also reformulate Theorem \ref{assumptionsformodel} for presheaf categories to obtain the results in this section.
\end{remark}

\subsection{Presheaves for the Coupled Model}

We are now interested in presheaves valued in coupled categories.  We will consider this from two different viewpoints. Let $\mathscr{X}$ be a simplicial Grothendieck topos. We define two categories
\begin{align*}
\mathbb{P}(\mathscr{X})&:= \textbf{Fun}(\mathscr{X}^{op},\mathscr{C}_F \mathscr{D})\\
\mathbb{Q}(\mathscr{X})&:= \textbf{Fun}(\mathscr{X}^{op},\mathscr{C})_{F^\ast}\textbf{Fun}(\mathscr{X}^{op},\mathscr{D})
\end{align*}
where $F^\ast \colon \textbf{Fun}(\mathscr{X}^{op},\mathscr{C}) \to \textbf{Fun}(\mathscr{X}^{op},\mathscr{D})$ is the induced functor on the presheaf categories by composition with $F$.

\begin{lemma}\label{reedylemma}
If the functor $F$ is Reedy admissible, then so is $F^\ast$ between the point-wise model structures on the presheaf categories.
\end{lemma}

\begin{proof}
As we will only be concerned with the resulting homotopy theory, it is enough to prove this for the injective point-wise model.  Assume that $\mathbb{A} \to \mathbb{A}'$ is a trivial cofibration in $\textbf{Fun}(\mathscr{X}^{op},\mathscr{C})$ and $F^\ast \mathbb{A}' \bigcup_{F^\ast \mathbb{A}} \mathbb{B} \to \mathbb{B}$ a trivial cofibration in  $\textbf{Fun}(\mathscr{X}^{op},\mathscr{D})$.  We need to show that $\mathbb{B} \to \mathbb{B}'$ is a weak equivalence in $\textbf{Fun}(\mathscr{X}^{op},\mathscr{D})$.  Recall that in the injective point-wise model structure, $\mathbb{A} \to \mathbb{A}'$ is a trivial cofibration if and only if $\forall x \in \mathscr{X}$ the induced map $\mathbb{A}(x) \to \mathbb{A}'(x)$ is a trivial cofibration in $\mathscr{C}$.  Likewise we consider $\forall x \in \mathscr{X}$ the maps
$$F^\ast \mathbb{A}' (x) \bigcup_{F^\ast \mathbb{A}(x)} \mathbb{B}(x) \to \mathbb{B}(x)$$
be trivial cofibrations in $\mathscr{D}$.  However note that $F^\ast \mathbb{A}'(x) = F(\mathbb{A}(x))$.  By assumption on $F$ it then follows that $\forall x \in \mathscr{X}$ the maps $\mathbb{B}(x) \to \mathbb{B}'(x)$ are weak equivalences in $\mathscr{D}$.  Therefore $\mathbb{B} \to \mathbb{B}'$ is a weak equivalence in the point-wise model structure for $\textbf{Fun}(\mathscr{X}^{op},\mathscr{D})$.
\end{proof}

\begin{proposition}\label{mainthrm}
The categories $\mathbb{P}(\mathscr{X})$ and $\mathbb{Q}(\mathscr{X})$ are isomorphic.  Moreover, there is a Quillen equivalence between the point-wise model structures. 
\end{proposition}

\begin{proof}
An object of $\mathbb{Q}(\mathscr{X})$ is of the form $p = (\mathbb{A} \colon \mathscr{X}^{op} \to \mathscr{C}, \mathbb{B} \colon \mathscr{X}^{op} \to \mathscr{D}, F^\ast \mathbb{A} \to \mathbb{B})$.  For $x \in \mathscr{X}$ we can look at $p(x) = (\mathbb{A}(x), \mathbb{B}(x), F(\mathbb{A}(x)) \to \mathbb{B}(x))$ which can be trivially viewed as an object of $\mathscr{C}_F\mathscr{D}$.  This assignment for each $x \in \mathscr{X}$ allows us to construct a functor $\alpha \colon \mathbb{Q}(\mathscr{X}) \to \mathbb{P}(\mathscr{X})$.  This functor can be seen as being fully faithful and essentially surjective due to its constructive nature.  The Quillen equivalence follows from comparing the classes of weak equivalences and fibrations in each category and showing that they coincide with the functor $\alpha$.
\end{proof}

Taking into account Proposition \ref{mainthrm}, we arrive at the following definition for presheaves with values in the coupled model structure.  We begin as always with the cyclic case.

\begin{definition}
A \textit{$\dC_N$-coupled presheaf} on a small Grothendieck site $\mathscr{C}$ is an object of
$$\textbf{sPr}(\mathscr{C})_{(\nabla^N)^\ast}\dC_N\textbf{-Pr}^s(\mathscr{C})$$
We will denote by $\dC_N\textbf{-Pr}^c(\mathscr{C})$ the category of all $\dC_N$ coupled presheaves on $\mathscr{C}$.
\end{definition}

Now we wish to equip this category with a local model structure.  To do this we will just change the presheaf categories appearing in the definition by their local models.  We list the following as a corollary since the existence of the model structures follows from Theorem \ref{coupledmodeltheorem} and a slight modification of Lemma \ref{reedylemma}.

\begin{corollary}
There is a cofibrantly generated model structure on $\dC_N\textbf{-Pr}(\mathscr{C})$ (resp., $\dD_N\textbf{-Pr}(\mathscr{C})$, $\dQ_M\textbf{-Pr}(\mathscr{C})$) defined as
\begin{align*}
& \textbf{sPr}_\tau(\mathscr{C})_{(\nabla^N)^\ast}\dC_N\textbf{-Pr}_\tau^s(\mathscr{C})\\
\text{(resp.,} \quad & \dR \textbf{-Pr}_\tau^s(\mathscr{C})_{(\Xi^{N})^{\ast}}\dD_N\textbf{-Pr}^s_\tau(\mathscr{C}) \text{)}\\
\text{(resp.,} \quad &\dR \textbf{-Pr}_\tau^s(\mathscr{C})_{(\Psi^{M})^{\ast}}\dQ_M\textbf{-Pr}^s_\tau(\mathscr{C}) \text{)}
\end{align*}
We will denote this model by $ \dC_N\textbf{-Pr}_\tau^c(\mathscr{C})$ (resp., $ \dD_N\textbf{-Pr}_\tau^c(\mathscr{C})$, $ \dQ_M\textbf{-Pr}_\tau^c(\mathscr{C})$) and call it the \textit{coupled local model structure}.
\end{corollary}

\section{Equivariant Sheaf Cohomology Theories}\label{sec5}

We finish this paper by discussing (equivariant) cohomology theories.  Recall from \cite{MR3285853} that for a site $\mathscr{C}$ the homotopy category $\text{Ho}(\textbf{sPr}_\tau(\mathscr{C}))$ is known as the \emph{category of $\infty$-stacks on $\mathscr{C}$}.  For a given model $m$ (weak, strong, coupled) we know that by Proposition \ref{gobject}, an object of $\text{Ho}(\dG \textbf{-Pr}^m_\tau(\mathscr{C}))$ can be seen as an $\infty$-stack on $\mathscr{C}$ along with some additional structure encoding the action of $\mathfrak{G}$.  Due to this reasoning, we make the following definition.

\begin{definition}
Let $\dG$ be a planar crossed simplicial group and $m$ one of the three model structures.  We will call  $\text{Ho}(\dG \textbf{-Pr}^m_\tau(\mathscr{C}))$ the \emph{category of $\dG^m$-equivariant $\infty$-stacks on $\mathscr{C}$}. 
\end{definition}

Of course, one can explore analogues of all homotopical algebraic constructs in this equivariant setting.  We choose to focus on the use of stacks in non-abelian cohomology theories. An overview of this for the non-equivariant case can be found in the article of To\"{e}n \cite{toenab}.  We will specialise to the setting of sheaf cohomology of sites, the following can be found in the book of Jardine \cite{MR3309296}.

\begin{definition}\label{ex:sheaf2}
Let $\mathscr{C}$ be a site and $\mathbb{H}(\mathscr{C}) := \text{Ho}(\textbf{sPr}_\tau\left(\mathscr{C}\right))$ the corresponding category of $\infty$-stacks.  Let $A$ be a sheaf of abelian groups. We then have $H^n(\mathscr{C};A) \simeq [\ast,K(A,n)]$ where $[-,-]$ denotes the set of morphisms in $\mathbb{H}(\mathscr{C})$ and $K(A,n)$ is an Eilenberg-MacLane object.
\end{definition}

\begin{example}
In the case that $\mathscr{C} \simeq \text{Op}(X)$, the site of open subsets of some topological space $X$, we retrieve the sheaf cohomology of $X$. 
\end{example}

We will now describe how the above construction works, as it will play a key role in the equivariant case. We  can see the above isomorphism as a chain of natural isomorphisms as follows:
\begin{align}
[\ast, K(A,n)] & \simeq [\tilde{\mathbb{Z}}\ast,K(A,n)] & \text{(simplicial abelian sheaves)}\label{eq:illusie}\\
& \simeq [\tilde{\mathbb{Z}}[0], A[-n]] & \text{(sheaves of chain complexes)}\label{eq:dold} \\
& \simeq\text{Ext}^n_\mathscr{C}(\mathbb{Z},A)\label{eq:final} \\
& \simeq H^n(\mathscr{C};A)\label{eq:final2}. 
\end{align}

Isomorphism \ref{eq:illusie} takes the setting from the homotopy category of simplicial presheaves to the corresponding homotopy category of simplicial abelian presheaves.  This step is an isomorphism because of the following result proved by Osdol \cite{MR0478160}.

\begin{proposition}[Illusie Conjecture]\label{prop:illusie}
The free simplicial abelian presheaf functor $X \mapsto \mathbb{Z}X$ considered as a map $\textbf{sPr}_\tau(\mathscr{C}) \to \textbf{sPr}_\tau(\mathscr{C})$ preserves local equivalences.
\end{proposition}

Isomorphism \ref{eq:dold} follows from the Dold-Kan correspondence.  This takes us into the setting of the homotopy category of simplicial abelian presheaves with morphisms now in the derived category.  Finally, isomorphisms \ref{eq:final} and \ref{eq:final2} can be taken simply as the corresponding definitions.

To mirror this construction in the equivariant case, we first need to consider what our coefficients will be.  In particular, we wish to define an equivariant Eilenberg-MacLane space. To do this, we revisit the construction of classical Eilenberg-MacLane spaces via the linearisation of spheres \cite[Example 1.14]{schwedebook2}.

We begin by recalling that the \emph{simplicial circle} is the simplicial set defined to be $S^1_\Delta := \Delta[1]/\partial\Delta[1]$ and the \emph{simplicial $n$-sphere} is defined to be the $n$-fold smash product: 
$$S_\Delta^n := \underbrace{S_\Delta^1 \wedge \cdots \wedge S_\Delta^1}_n.$$
We can then consider for any abelian group $A$, the space $A \otimes \mathbb{Z}[S_\Delta^n]$ where the second term is the free abelian group generated by the level-wise non-base-points of $S_\Delta^n$.  That is, we can see the points of $A \otimes \mathbb{Z}[S_\Delta^n]$ to be finite sums of points in $S_\Delta^n$ with coefficients in $A$ modulo the relation that all $A$-multiples of the base-point are zero.  One can equip this space with the induced quotient topology where $A$ is given the discrete topology.

\begin{proposition}[{\cite[Corollary 6.4.23]{MR1908260}}]
For any abelian group $A$ the space $A \otimes \mathbb{Z}[S_\Delta^n]$ is a $K(A,n)$ space.  Note that for $n \geq 1$ the space has a unique connected component and therefore we do not need to specify a base-point.
\end{proposition}

As we have defined the above using a simplicial construction, it can very easily be modified to the $\dG$ setting.  We begin by defining the analogue of the simplicial sphere.

\begin{definition}\label{index:equicircle}
The \emph{$\dG$-circle} is the $\dG$-set defined to be $S_\mathfrak{G}^1 := \dG[1] / \partial \dG[1]$.  The \emph{$\dG$-$n$-sphere} is the the $n$-fold smash product of $S_\mathfrak{G}^1$.
\end{definition}

\begin{lemma}
$|S_\mathfrak{G}^1|_\mathfrak{G} \simeq |\mathfrak{G}| \times SO(2)$.
\end{lemma}

\begin{proof}
$$|S_\mathfrak{G}^1|_\mathfrak{G} \simeq |\mathfrak{G} \times \Delta[1]/\partial \Delta[1]| \simeq |\mathfrak{G}| \times |\Delta[1]/\partial \Delta[1]| \simeq |\mathfrak{G}| \times SO(2).$$
\end{proof}

\begin{definition}\label{index:equimaclane}
Let $A$ be an abelian group, we define the $\dG$-$K(A,n)$ object to be
$$K^\mathfrak{G}(A,n) := A \otimes \mathbb{Z}[S_\mathfrak{G}^n],$$
where $ \mathbb{Z}[S_\mathfrak{G}^n]$ is the abelian group freely generated by the level-wise non-base-point elements.
\end{definition}

By considering the underlying simplicial object $i^\ast(K^\mathfrak{G}(A,n)) \simeq K(A,n)$ we see that such an object is both $n$-connected and $n$-truncated as we would expect.

We now look at altering the site cohomology from the simplicial case to the equivariant case using the equivariant Eilenberg-MacLane objects.

\begin{definition}\label{index:equicohomology}
Let $\mathscr{C}$ be a site and $\mathbb{H}^m_\mathfrak{G}(\mathscr{C}):= \text{Ho}( \dG \textbf{-Pr}_\tau^m\left(\mathscr{C}\right))$ the corresponding equivariant topos where $m$ is one of the model choices (weak, strong, coupled).  Let $A$ be a sheaf of abelian groups. We define the $\dG^m$-cohomology of $\mathscr{C}$ to be
$$H^n_{\mathfrak{G},m}(\mathscr{C};A) := [\ast,K^\mathfrak{G}(A,n)]$$
\end{definition}

Clearly the choice of model structure plays a part in the cohomology defined above.  The weak model structure is the coarsest model, and retrieves Borel style cohomology, whereas the strong and coupled model are finer, and are of Bredon type.  This construction has been done using presheaves of $G$-spaces, in contrast, what we have done here is a combinatorial discrete formalisation \cite{equimodels}.

We will unpack this definition for the cyclic case.  We work with the cyclic case as we wish to produce a series of natural isomorphisms as in the non-equivariant case, and all such isomorphisms can be described using work that has been done in the cyclic setting.  We begin by proving the equivariant version of the Illusie conjecture, for which we must consider the category $\dC \textbf{-R-mod}$ of $\dC$-objects in the category of $R$-modules.  There is a free functor $C : \dC \textbf{Set} \to \dC \textbf{-R-mod}$, arising as a left adjoint to the forgetful functor going the other way.  The category $\dC\textbf{-R-mod}$ can be equipped with the corresponding weak/string/coupled model via a transfer.  We promote the free functor $C$ to $\dC$-presheaves by applying it section-wise.  Using this, it is possible to move from the setting of $\dC$-presheaves to $\dC$-abelian sheaves:
$$[\ast,K^\mathfrak{G}(A,n)] \to [C \ast, K^\mathfrak{G}(A,n)]$$
and in fact, this association is an isomorphism due to the following proposition.

\begin{proposition}[Equivariant Illusie Conjecture]
The free $\dG$-abelian presheaf functor $X \mapsto CX$ considered as a map $\dC \textbf{-Pr}^m_\tau(\mathscr{C}) \to \dC \textbf{-Pr}^m_\tau(\mathscr{C})$ preserves local equivalences.
\end{proposition}

\begin{proof}
Recall from the previous section that the strong (and indeed the weak and coupled) models rely on the subdivision functors     $\Phi_r : \dC\textbf{-Set} \to \textbf{sSet}$.  We will prove the result for the case of the strong model, the weak and coupled models follow easily from this. First we note that for any subdivision $\Phi_r CX \simeq C \Phi_r X$.  Now we note that a map $f$ is a weak equivalence in $\dC \textbf{-Pr}_\tau^s (\mathscr{C})$ if and only if $\Phi_r f$ is a weak equivalence in $\textbf{sPr}_\tau(\mathscr{C})$ for all $r \geq 0$.  Therefore our desired result follows from the simplicial Illusie conjecture, Proposition \ref{prop:illusie}.
\end{proof}

\begin{remark}
Although we have shown the above only for the cyclic case, we can see that it holds for all of the planar crossed simplicial groups as the subdivision functors all have the required property that we use in the proof.
\end{remark}

We have now replicated the first step from the simplicial case. Recall that the next step was the isomorphism $[\tilde{Z}\ast,K(A,n)] \simeq [\tilde{Z}[0],A[-n]]$ from the maps in the homotopy category of simplicial abelian sheaves to maps between sheaves of chain complexes which was facilitated using the Dold-Kan correspondence.  In the general equivariant case, such an isomorphism is unaccessible as we would need some Dold-Kan style theorem between the abelian objects and certain ``chain complexes''.  However, in the cyclic case such a construction exists as a consequence of work on so-called duchain complexes (see \cite{cyclic,DWYER1987165,MR3320603}).  We will not introduce the full theory theory but just give the relevant result.  First, we introduce what we mean by \emph{cyclic chain} and \emph{duchain complexes}.

\begin{definition}
A \emph{duchain complex} over $R$ is a diagram of $R$-modules
$$\xymatrix{X_0 \ar@<-0.5ex>[r]_{\delta_0} & X_1 \ar@<-0.5ex>[r]_{\delta_1} \ar@<-0.5ex>[l]_{\partial_1} &  X_2 \ar@<-0.5ex>[r]_{\delta_2} \ar@<-0.5ex>[l]_{\partial_2} &  \cdots  \ar@<-0.5ex>[l]_{\partial_3}}$$
such that $\partial^2 = \delta^2 = 0$, but otherwise the $\partial$'s and $\delta$'s are independent. A map of duchain complexes is a set of maps $f_i \colon X_i \to Y_i$ such that all the obvious squares commute.

Denote by $f_i(t) = \left( 1+ (-1)^i t \right)^{i+1}$ for $i \geq 0$.  A \emph{cyclic chain complex} is a duchain complex such that
\begin{itemize}
\item $f_{n-1}f_n(\delta \partial)x=x$, for all $x \in X_n$, $n>0$.
\item $f_0(\delta \partial)x =x$. for all $x \in X_0$.
\end{itemize}
\end{definition}

We can now use this notion of cyclic complexes to give the statement of the ``cyclic Dold-Kan theorem''.

\begin{proposition}[{\cite[Proposition 2.2]{DWYER1987165}}]
The category of cyclic chain complexes over $R$ is equivalent to the category of $\dC$-$R$-modules.
\end{proposition}

Using the above proposition we see that we have therefore got an isomorphism to a mapping between sheaves of cyclic complexes as one would expect. In the simplicial case we had
$$ [\tilde{\mathbb{Z}}[0], A[-n]] \simeq\text{Ext}^n_\mathscr{C}(\mathbb{Z},A) \simeq H^n(\mathscr{C};A).$$
In the equivariant case the choice of model will determine what we get at this step.  What can be said is that in the weak model case, Dwyer and Kan proved that the category of $R$-modules equipped with the weak model is Quillen equivalent to another category, namely the category of \emph{mixed complexes} \cite{DWYER1987165}.  And then in \cite{MR883882} Kassel shows that there is a way to compute cyclic cohomology as derived functors on a mixed complex using an Ext construction.  Therefore, we see that in the weak cyclic case, we retrieve a cyclic cohomology, which can be seen as a Borel type equivariant cohomology theory as shown in \cite{MR870737}.  We conclude that the definition of equivariant site cohomology via equivariant $\infty$-stacks is a sensible one.  The advantage of using this definition of equivariant cohomology is that we do need need results about Dold-Kan type correspondences to compute it for a general planar Lie group, and we have control over the level of equivariance via a choice of model structure.

\bibliographystyle{siam}
\bibliography{phdbib}

\end{document}